\documentclass[11pt,twoside]{amsart}

\usepackage{latexsym}
\usepackage{amsmath}
\usepackage{amsthm}
\usepackage{amssymb}
\usepackage{vmargin}
\usepackage{amscd}
\usepackage{stmaryrd}
\usepackage{euscript}
\usepackage{mathrsfs}
\usepackage{amscd}
\usepackage[all]{xy}
\usepackage{xr}

\DeclareMathAlphabet{\mathpzc}{OT1}{pzc}{m}{it}

\newtheorem{theorem}{Theorem}[section]
\newtheorem{proposition}[theorem]{Proposition}

\newtheorem{lemma}[theorem]{Lemma}
\newtheorem*{theorem*}{Theorem}
\newtheorem*{proposition*}{Proposition}
\newtheorem*{corollary*}{Corollary}
\newtheorem*{lemma*}{Lemma}
\newtheorem*{conjecture*}{Conjecture}

\theoremstyle{definition}
\newtheorem{definition}[theorem]{Definition}

\newtheorem*{definition*}{Definition}

\theoremstyle{remark}

\newtheorem{remark}[theorem]{Remark}

\newtheorem*{example*}{Example}
\newtheorem*{examples*}{Examples}
\newtheorem*{remark*}{Remark}
\newtheorem*{remarks*}{Remarks}
\newtheorem*{exercise*}{Exercise}
\newtheorem*{property*}{Property}
\newtheorem*{properties*}{Properties}

\newcommand\ten{\otimes}

\renewcommand\H{\mathrm{H}}

\newcommand\Z{\mathbb{Z}}
\newcommand\Q{\mathbb{Q}}

\newcommand\bL{\mathbb{L}}

\newcommand\bS{\mathbb{S}}

\newcommand\C{\mathcal{C}}

\newcommand\cD{\mathcal{D}}

\newcommand\cM{\mathcal{M}}
\newcommand\cN{\mathcal{N}}

\newcommand\cS{\mathcal{S}}

\newcommand\cZ{\mathcal{Z}}

\renewcommand\O{\mathscr{O}}

\newcommand\sM{\mathscr{M}}

\newcommand\sT{\mathscr{T}}

\newcommand\fL{\mathfrak{L}}

\newcommand\fX{\mathfrak{X}}
\newcommand\fY{\mathfrak{Y}}
\newcommand\fZ{\mathfrak{Z}}

\newcommand\fLog{\mathfrak{Log}}

\newcommand\rC{\mathring{C}}

\newcommand\rU{\mathring{U}}
\newcommand\rV{\mathring{V}}

\newcommand\rX{\mathring{X}}

\newcommand\cHom{\mathcal{H}\!\mathit{om}}

\newcommand\Alg{\mathrm{Alg}}

\newcommand\Hom{\mathrm{Hom}}

\newcommand\map{\mathrm{map}}
\newcommand\HHom{\underline{\mathrm{Hom}}}

\newcommand\Spec{\mathrm{Spec}\,}

\newcommand\LLim{\varinjlim}

\newcommand\ho{\mathrm{ho}\!}

\newcommand\pr{\mathrm{pr}}

\newcommand\bt{\bullet}
\newcommand\by{\times}

\newcommand\diag{\mathrm{diag}\,}

\newcommand\gp{\mathrm{Gp}}

\newcommand\co{\colon\thinspace}

\newcommand\oR{\mathbf{R}}

\newcommand\oL{\mathbf{L}}

\newcommand\uleft\underleftarrow
\newcommand\uline\underline
\newcommand\uright\underrightarrow

\begin{document}

\begin{abstract}
In this note, we give a formulation of log structures for derived stacks using Olsson's log stack. 
 The derived cotangent complex is then Olsson's logarithmic cotangent complex, which (unlike Gabber's) is just given by log differential forms in the log smooth case. The derived moduli stack of log stable maps  then produces the desired  virtual tangent space and obstruction theory on the underlying underived stack.
\end{abstract}

 
\title{Derived log stacks after Olsson}
\author{J.P.Pridham}
\thanks{This work was supported by the Engineering and Physical Sciences Research Council [grant numbers  EP/F043570/1 and EP/I004130/1].}
\maketitle



\tableofcontents

\section{Olsson's log stacks} 

\begin{definition}
Recall from \cite[\S 2]{olssonlogcot} that the  algebraic stack $\fL^0$ classifies fine log structures. In other words, for any scheme $U$, $\fL^0(U)$ is the groupoid of fine log structures on $U$. 

The algebraic stack $\fL^1$ classifies morphisms of fine log structures, so $\fL^1(U)$ is the groupoid of pairs $(\sM_0, \sM_1)$ of log structures on $U$ with a morphism $\sM_0 \to \sM_1$. The algebraic stacks  $\fL^n$ classify n-strings of morphisms of fine log structures, so $\fL^n(U)$ is the groupoid of strings 
\[
 \sM_0 \to \sM_1 \to \sM_2 \to \ldots \to \sM_n.
\]

More generally, for a finite category $\Gamma$, the algebraic stack $\fL^{\Gamma}$ classifies $\Gamma$-diagrams of log structures.
\end{definition}

Beware that these stacks are neither quasi-compact nor quasi-separated.

\begin{definition}
For $0<j< n$, we define $\delta_j^* \co \fL^{n} \to \fL^{n-1}$ to be given by composition at the $j$th place, sending 
\[
 \sM_0 \to \sM_1 \to \sM_2 \to \ldots \to \sM_{n}
\]
to
\[
 \sM_0 \to  \ldots \to  \sM_{j-1} \to \sM_{j+1} \to \ldots \to \sM_{n}.
\]
The map $\delta_0^*$ is defined by forgetting $\sM_0\to$, and the map $ \delta_n^*$ by forgetting $\to \sM_n$.

When no map is specified, we adopt the convention that $ \fL^{n} \to \fL^{n-1}$ is understood to be the map $\delta_n^*$.
\end{definition}

\begin{definition}\label{fLogS}
Given a log-flat log scheme $S$, 
recall from \cite{olssonLog} that the  algebraic stack $\fLog_S$ classifies fine log schemes. In other words, for any affine scheme $U$, $\fLog_S(U)$ is the groupoid of fine log schemes over $S$ with underlying scheme $U$. 
\end{definition}

Writing $\mathring{S}$ for the scheme underlying a log scheme $S$, the log structure can be interpreted as a morphism $\mathring{S} \to \fL^0$, and then \cite[\S 2]{olssonlogcot} observes that 
\[
 \fLog_S= \fL^1\by_{\fL^0}\mathring{S},
\]
with a map $U \to \fLog_S$ corresponding to a morphism $f^*\sM_S \to \sM_U$ of log structures on $U$, for $f\co U \to \mathring{S}$.

\begin{lemma}\label{logsmooth}
 A morphism $X \to S$ of fine log schemes is log smooth if and only if $\rX \to \fLog_S$ is smooth.
\end{lemma}
\begin{proof}
 This is part of \cite[Theorem 4.6]{olssonLog}.
\end{proof}

As we will only be interested in locally free log structures, we make the following definition:
\begin{definition}
 Define $\fL^n_f \subset \fL^n$ to be the open substack parametrising diagrams of locally free log structures.
\end{definition}
 
\begin{definition}
 Given a monoid $P$, we follow \cite{olssonLog} in writing $P^{\gp}$ for the group completion of $P$.
\end{definition}

\begin{proposition}
 The algebraic stacks $\fL^n_f $ are smooth for all $n$.
\end{proposition}
\begin{proof}
By \cite[Corollary 5.25]{olssonLog}, the stack $\cS_P:= [S[P]/S[P^{\gp}]]$ has the property  that if $S$ has a local chart $Q \to \sM_S|_U$ on   $U$, then 
\[
 \coprod_{Q \to P} \cS_P\by_{\cS_Q}U\to \fLog_S
\]
is \'etale and surjective whenever $U\to \mathring{S}$ is so, where the coproduct is taken over morphisms of finitely generated fine saturated monoids.   

Writing $\cZ_P:= [\Spec \Z[P]/\Spec \Z[P^{\gp}]]$, it follows that  $\coprod_Q \cZ_Q\to \fL^0\simeq \fLog_{(\Spec\Z, \O^*)}$ is an \'etale surjection,  and (since $\fL^1\simeq \fLog_{\fL^0}$), it also follows that
\[
\coprod_{Q \to P} \cZ_P\to \fL^1
\]
is an \'etale surjection.

In general, we have $\fL^n\simeq \fLog_{\fL^{n-1}}$, giving us an \'etale surjection  
\[
 \coprod_{P_0 \to \ldots \to P_n} \cZ[P_n] \to \fL^n.
\]
If $P$ is a free monoid, then $\Spec \Z[P]$ is smooth and an atlas for $\cZ[P]$ (the latter following because $ \Spec \Z[P^{\gp}]$ is also smooth). Thus
\[
 \coprod_{P_0 \to \ldots \to P_n} \Spec \Z[P_n] \to \fL^n_f
\]
 is a smooth atlas, where the coproduct is now taken over morphisms of finitely generated free monoids $P_0, \ldots , P_n$.
\end{proof}

\begin{remark}\label{finalobjrmk}
 In fact, an easy generalisation of \cite[Corollary 5.25]{olssonLog} shows that if a finite category $\Gamma$ has final object $\gamma$, then there is an \'etale surjection
\[
 \coprod_{\uline{P}} \cZ_{\uline{P}(\gamma)}\to \fL^{\Gamma},
\]
where $\uline{P}$ runs over all functors from $\Gamma$ to finitely generated fine saturated monoids.
\end{remark}

\begin{definition}\label{fLognfdef}
 Given a log-flat log scheme $S$, define $\fLog^n_{S,f}:= \fL^{n+1}_f \by_{\fL^0_f}S$.
\end{definition}

\begin{lemma}\label{etalemaplemma}
 For $j <n $, the map $\delta_j^*\co \fL^n \to \fL^{n-1} $ is relatively Deligne--Mumford and \'etale.
\end{lemma}
\begin{proof}
This is given in  \cite[Prop 2.11]{olssonlogcot}.
\end{proof}

In fact, Remark \ref{finalobjrmk} ensures that $\fL^{\Delta} \to \fL^{\Gamma}$ is relatively Deligne--Mumford and \'etale  and \'etale whenever $\Gamma$ and $\Delta$ both have final objects and the functor $\Gamma \to \Delta$ preserves the final object.

Beware that Lemma \ref{etalemaplemma} does not apply to $\delta_n^*$, which is the morphism we use most frequently.

\section{Log structures on derived stacks}

Denote the category of simplicial sets by $\bS$.

There is a natural embedding of the homotopy category of stacks into the homotopy category of derived stacks, and we essentially now just define a derived log stack $\fX$ to be  a derived stack $\mathring{\fX}$ equipped with a morphism to $\fL^0$. A morphism $\fX \to \fY$ of derived log stacks is then a commutative diagram
\[
 \begin{CD}
  \mathring{X} @>>> \fL^1 \\
@VVV @VVV\\
\mathring{Y} @>>> \fL^0,
 \end{CD}
\]
and so on.

However, for a singular stack, the associated derived stack tends to be fairly unnatural, so we restrict to the smooth stacks $\fL^n_f$ and consider only  locally free log structures on derived stacks. We now recall from \cite{drep} an explicit description of the derived stack associated to a smooth stack.

\begin{definition}
For a commutative ring $R$, define $s\cN_R^{\flat}$ to be the category of simplicial commutative  $R$-algebras $A$ for which 
\begin{enumerate}
 \item the map $A \to \pi_0A$ has nilpotent kernel;
\item the Dold--Kan normalisation $NA$ is bounded (i.e. $N_iA=0$ for all $i \gg 0$).
\end{enumerate}
\end{definition}

\begin{definition}
Given
a functor $F: \Alg_R \to \bS$ on $R$-algebras, we  define  a functor $\uline{F}: s\cN^{\flat}_R\to s\bS$ to  the category of bisimplicial sets by 
$$
\underline{F}(A)_{n} :=  F(A_n).
$$
\end{definition}
By \cite[Theorem \ref{drep-lurierep3}]{drep}, a derived geometric stack over $R$ is determined by its restriction to $s\cN^{\flat}_R$.

\begin{proposition}\label{ulines}
 Given a smooth algebraic stack $\fX$ over $R$, the associated derived stack  is given on $s\cN^{\flat}_R$ by the diagonal 
$\diag B\uline{\fX}$ of the nerve of $\uline{\fX}$. Explicitly,
\[
 \diag B\uline{\fX}(A)_n= (B_n \fX)(A_n).
\]
\end{proposition}
\begin{proof}
 This is a straightforward consequence of \cite[Theorem \ref{drep-lurieprerep}]{drep}.
\end{proof}

Note that $\diag B\uline{\fX}(A)$ is just a model for the homotopy colimit $\ho\LLim_n B\fX(A_n)$, so we could use other model (such as $\bar{W}B\uline{\fX}(A)$ for the codiagonal $\bar{W}$ of \cite{CRdiag}) for the homotopy colimit. 

In the setting of \cite{hag2}, we can form the simplicial presheaf $B\fX$ on $s\cN^{\flat}_R$ by $B\fX(A):= B\fX(A_0)$, and then Proposition \ref{ulines} says that  $\diag B\uline{\fX}$ is the $\infty$-stackification $(B\fX)^{\sharp} $ of  $B\fX$.

\begin{definition}
 For a commutative $\Q$-algebra $R$, define $dg_+\cN_R^{\flat}$ to be the category of differential graded-commutative $R$-algebras in non-negative chain degrees 
for which 
\begin{enumerate}
 \item the map $A_0 \to \H_0A$ has nilpotent kernel;
\item $A$ is bounded (i.e. $A_i=0$ for all $i \gg 0$).
\end{enumerate}
\end{definition}

\begin{definition}
 Set $\Omega_n=\Omega(\Delta^n)$ to be the graded-commutative  chain algebra 
$$
\Q[t_0, t_1, \ldots, t_n,dt_0, dt_1, \ldots, dt_n ]/(\sum t_i -1, \sum dt_i)
$$  
of rational differential forms on the $n$-simplex $\Delta^n$, with $t_i$ in degree $0$.
 These fit together to form a simplicial diagram $\Omega_{\bt}$ of chain algebras, and for a chain algebra $A$, we define $A^{\Delta^n}$ as the good truncation $A^{\Delta^n}:= \tau_{\ge 0}(A \ten \Omega_n)$.
\end{definition}

\begin{definition}
Given
a functor $F: \Alg_R \to \bS$, we  define  a functor $\uline{F}: dg_+\cN^{\flat}\to s\bS$ to  the category of bisimplicial sets by 
$$
\uline{F}(A)_{n} :=  F((A^{\Delta^n})_0).
$$
\end{definition}

By \cite[Theorem \ref{drep-lurierep3}]{drep}, a derived geometric stack over a $\Q$-algebra $R$ is determined by its restriction to $dg_+\cN^{\flat}_R$.
\begin{proposition}\label{ulinedg}
 Given a smooth algebraic stack $\fX$ over a $\Q$-algebra $R$, the associated derived stack is given on $dg_+\cN^{\flat}_R$ by  $\diag B\uline{\fX}$  (or equivalently $\bar{W}B\uline{\fX}$). 
\end{proposition}
\begin{proof}
 Again, this is a straightforward consequence of \cite[Theorem \ref{drep-lurieprerep}]{drep}.
\end{proof}

\begin{definition}\label{pi0def}
 Given a derived $\infty$-stack $F \co dg_+\cN^{\flat}_R \to \bS$ or $F \co s\cN^{\flat}_R \to \bS$, define the underlying underived  $\infty$-stack $\pi^0F\co \Alg_R \to \bS$ by $\pi^0F(A):= F(A)$.
\end{definition}

\begin{definition}
 Define a (locally free) log structure on a derived $\infty$-geometric stack $\fX$ to be a morphism $\fX \to \bar{W} B\uline{\fL^0_f}$ of derived log stacks. A morphism of  log structures on a derived $n$-geometric stack $\fX$  is then defined to be a morphism $\fX \to \bar{W}B\uline{\fL^1_f}$.
\end{definition}
From now on, any log structure on a derived $\infty$-geometric stack is understood to be locally free.

\begin{remark}\label{inftycatrmk}
A homotopy class of log structures on a derived $\infty$-geometric stack $\fX$ will be a homotopy class of morphisms $\fX \to \bar{W}B\uline{\fL^0_f}$, but we should really regard the simplicial set
\[
 \HHom(\fX, \bar{W}B\uline{\fL^0_f} )\simeq \map(\fX,B\fL^0_f )
\]
as the $\infty$-groupoid of locally free log structures on $\fX$, so $\pi_0$ gives the homotopy classes of objects,   $\pi_1$ the homotopy classes of automorphisms, and so on. 

Of course, we are interested in all morphisms of log structures, not just automorphisms, so we should look at the bisimplicial set
\[
 n \mapsto  \HHom(\fX, \bar{W}B\uline{\fL^n_f}),
\]
which is a complete Segal space (in the sense of \cite{rezk}) and hence a model for an $(\infty,1)$-category. This gives us an object we can regard as the $\infty$-category of log structures on $\fX$. The reason the Segal space is complete is that 
the full substack of $\fL^1$ parametrising isomorphisms of log structures is equivalent to $\fL^0$.

In order to understand the difference between this $\infty$-category and the $\infty$-category of log structures considered in \cite{SagaveSchuergVezzosi} (which was based on Gabber's log cotangent complex), we can look at connected components of the $\infty$-groupoid of log structures on a derived affine scheme. For $A \in s\cN^{\flat}_{\Z}$, a homotopy class of log structures on $\oR \Spec A$ is an element of
the set  $\pi_0\uline{\fL^0_f}(A)$. This is given by a log structure on $A_0$ modulo isomorphisms and homotopies coming from log structures on $A_1$. In the setting of \cite{SagaveSchuergVezzosi}, a homotopy class of log structures on $\oR \Spec A$ is given by a simplicial log structure on $A$, modulo isomorphisms and simplicial log structures on $A^{\Delta^1}$.
\end{remark}

\begin{definition}
Write $\oR\fL^n$ for the derived stackification $\bar{W}B\uline{\fL^n_f}$ of the smooth stack $\fL^n_f$.

Given a derived log stack $V$, define $\oR\fLog_{V}$ to be the homotopy fibre product $\oR\fLog_{V}:= \mathring{V}\by^h_{\oR\fL^0_f}\oR\fL^1_f$. Note that when $V$ is a log-flat underived stack, this agrees with Definition \ref{fLogS},  since $\oR\fLog_{V} $ is just the derived stackification of $\fLog_{V,f}$.  Likewise, define $\oR\fLog^n_{V}:= \mathring{V}\by^h_{\oR\fL^0_f}\oR\fL^{n+1}_f$.
\end{definition}
Thus working with derived stacks is one way to resolve the base change issues arising for Olsson's log stack when the base is not log-flat.

\begin{definition}
 Given a morphism $\fX \to \fY$ of derived  $\infty$-geometric log stacks, write
\[
 \bL^{\fX/\fY}:= \bL^{\mathring{\fX}/\oR\fLog_{\fY}}
\]
for the logarithmic cotangent complex. Note that this agrees with \cite{olssonlogcot} when $\fX$ and $\fY$ are both log-flat underived log stacks.
\end{definition}

\begin{remark}\label{triangle}
By Lemma \ref{etalemaplemma}, the map $\delta_0^* \co \fLog^1_{\fY} \to \fLog_{\fY}$ is \'etale, so  
\[
 \bL^{\oR\fLog_{X}/ \oR\fLog_{\fY}} \simeq \bL^{\oR\fLog_{X}/ \oR\fLog_{\fY}^1},
\]
which is the pullback along $\fLog_{\fY}^1 \to \fLog_{\fY}^0$ of $\bL^{\fX/\fY}$. 

Thus for any diagram $\fX \to \fY \to \fZ $ of derived $\infty$-geometric log stacks,  the diagram $\mathring{\fX} \to \oR\fLog_{\fY} \to \oR\fLog_{\fZ}$  gives a distinguished  triangle
\[
\oL f^*\bL^{\fY/\fZ}\to \bL^{\fX/\fZ} \to \bL^{\fX/\fY}\to \oL f^*\bL^{\fY/\fZ}[1]. 
\]
\end{remark}

\section{Logarithmic derived moduli stacks}

We now fix a base log stack $V$, 
 with locally free log structure, and consider moduli of log curves over $V$, and then moduli of log maps from log curves to a fixed target, as features in the study of stable maps. The universal case takes $\mathring{V} = \fL^0_f$, with log structure given by the identity map $\fL^0_f \to \fL^0_f$. 

\subsection{Moduli of log curves}

\subsubsection{The underived moduli stack}

We first consider moduli of log curves. For any affine log scheme $U$ over $V$ with locally free log structure, we wish to parametrise proper log-smooth log curves over $U$. Now, an affine log scheme $U$ over $V$ is the same as an affine scheme $\rU$ equipped with a morphism
\[
 \rU \to \fLog_{V,f}.
\]
Given such data, a  log-smooth log curve over $U$ with locally free log structure is smooth morphism 
\[
 \rC \to \fLog_{U,f}
\]
of relative dimension $1$. 

If $\cM_g'$ denotes the (non-algebraic) moduli stack of smooth, not necessarily proper, genus $g$ curves, then the groupoid of log curves above is equivalent to the groupoid of stack homomorphisms
\[
 \fLog_{U,f} \to \cM_g'.
\]
Since $\fLog_{U,f}= \fLog^1_{V,f}\by_{\fLog_{V,f}}\rU$, this is the same as a map from $\rU$ to the $\Hom$-stack
\[
 \mathring{\cN}_{g,V}':=\cHom_{\fLog_{V,f}}( \fLog^1_{V,f}, \cM_g' \by \fLog_{V,f}).
\]
Projection to $\fLog_{V,f} $ gives us a log stack $ \cN_{g,V}'$ with underlying stack $\mathring{\cN}_{g,V}'$.

\begin{remark}
In what follows, the only properties of families  $f \co C \to S$ of curves which we will use  are smoothness of $f$ and that $\oR f_*$ maps quasi-coherent sheaves to complexes of length $1$. Thus everything can be generalised to stacky curves.
\end{remark}

We now define
\[
 \mathring{\cN}_{g,V} \subset \mathring{\cN}_{g,V}'
\]
to consist of those $\rC \to \fLog_{U,f}$ for which $\rC \to \rU$ is proper. Because properness is deformation invariant, this is an open substack.  Note that this substack is not the same as
\[
 \cHom_{\fLog_{V,f}}( \fLog^1_{V,f}, \cM_g \by \fLog_{V,f}),
\]
because properness of $\rC \to \rU $ does not imply properness of $ \rC \to \fLog_{U,f}$.  Write $\cN_{g,V} $ for the log stack over $V$ given by the morphism $\mathring{\cN}_{g,V} \to \fLog_{V,f}$. 

We simply write $\cN_g:= \cN_{g,\fL^0_f}$ when working with the universal base $\fL^0_f$. Note that   $\mathring{\cN}_{g,V}= \mathring{\cN}_{g}\by_{\fL^0_f}V$.

Since stability is invariant under deformations,  the log  moduli stack of stable log curves over $V$  forms an open substack of $\cN_{g,V} $. We make no claims about algebraicity of this stack, but note that formal smoothness of $\cM_g'$
ensures that the morphism
\[
 \mathring{\cN}_{g,V}' \to \fLog_{V,f}
\]
 is formally smooth, hence  $\mathring{\cN}_{g,V} \to \fLog_{V,f}$ is also formally smooth. 

Moreover, the stacks $\mathring{\cN}_g'$ and $\mathring{\cN}_g $ are homogeneous. This means that
for any nilpotent extension $A \to B$ of rings, and any morphism $C \to B$,  the maps
\[
 \mathring{\cN}_g'(A\by_BC) \to \mathring{\cN}_g'(A)\by_{\mathring{\cN}_g'(B)}\mathring{\cN}_g'(C), \quad \mathring{\cN}_g(A\by_BC) \to \mathring{\cN}_g(A)\by_{\mathring{\cN}_g(B)}\mathring{\cN}_g(C)
\]
are equivalences of groupoids. 

For any open substack of $\mathring{\cN}_g$ (and in particular for the moduli stack of log stable curves), this means that satisfying Artin's conditions for algebraicity is just a matter of verifying formal existence, local finite presentation and finiteness of tangent spaces.

\subsubsection{The universal curve}

A map $\rU \to \mathring{\cN}_g'$ necessarily gives rise to a log curve $C$ over $U$, so there is a universal such curve over $\mathring{\cN}_g'$.
If we write $\C_g' \to \cM_g'$ for the universal curve over $\cM_g'$, then  $\rC$ is formed as the pullback
\[
 \begin{CD}
 @. \rC @>>> \C_g' \\
@. @VVV @VVV \\
\fLog^1_{V,f}\by_{\fLog_{V,f}} \rU @=\fLog_{U,f} @>>> \cM_g'.
 \end{CD}
\]

The universal case is given by taking $\rU= \mathring{\cN}_{g,V}'$, giving a universal curve $\mathring{\cD}_{g,V}'$ as the pullback
\[
\begin{CD}
  \mathring{\cD}_{g,V}' @>>> \C_g' \\
@VVV @VVV \\
\fLog^1_{V,f}\by_{\fLog_{V,f}} \mathring{\cN}_{g,V}' @>>> \cM_g',
 \end{CD}
\]
with log stack $\cD_{g,V}'$ over $\cN_{g,V}'$ given by the map from $\mathring{\cD}_{g,V}'$ to  
\[
 \fLog_{ \mathring{\cN}_{g,V}',f}= \fLog^1_{V,f}\by_{\fLog_{V,f}}\mathring{\cN}_{g,V}'= \fLog^1_{V,f}\by_{\fLog_{V,f}} \cHom_{\fLog_{V,f}}( \fLog^1_{V,f}, \cM_g' \by \fLog_{V,f}).
\]
We then construct the universal curve $\cD_{g,V}$ over $\cN_{g,V}$ as the pullback
\[
 \cD_{g,V}:= \cD_{g,V}'\by_{ \cN_{g,V}'}\cN_{g,V}.
\]

\subsubsection{Derived moduli and the tangent space}

We may also construct a derived moduli stack by the same method. Since $\cM'_g$ is a formally smooth stack, it gives rise to a derived stack $\oR \cM_g':=\bar{W}B\uline{\cM'}_g$ defined as in Proposition \ref{ulines} or \ref{ulinedg}.

We then set
\[
 \oR\mathring{\cN}_{g,V}':=\oR\cHom_{\oR\fLog_{V}}( \oR\fLog^1_{V}, \oR\cM_g' \by \oR\fLog_{V}),
\]
which is a derived $\Hom$-stack with underlying stack $\pi^0\oR\mathring{\cN}_{g,V}'= B\mathring{\cN}_{g,V}'$.
Then take $\oR\mathring{\cN}_{g,V} \subset \oR\mathring{\cN}_{g,V}' $ to be the open derived substack on the underlying stack $B\mathring{\cN}_{g,V}$. In other words
\[
 \oR\mathring{\cN}_{g,V}(A):=  \oR\mathring{\cN}_{g,V}'(A) \by^h_{B\mathring{\cN}_{g,V}'(\H_0A) }B\mathring{\cN}_{g,V}(\H_0A).
\]

For open substacks $\fX$ of $\mathring{\cN}_{g,V}$, such as derived moduli of stable log curves, this also allows us to define a derived stack $\oR \fX$ by the same formula.

For a derived affine scheme $U$ over $\oR\fLog_{V}$, we can reinterpret  $\oR\mathring{\cN}_{g,V}'(U)$  as the $\infty$-groupoid of smooth genus $g$ curves  $\rC \to \oR\fLog_{U}$.  This follows essentially because $\oR\cM_g'$ is the derived moduli stack of smooth genus $g$ curves, since  the cotangent complex of a smooth morphism has no higher terms and a curve has no cohomology above degree $1$ (see \cite[\S 8]{lurie} for details). 

\begin{proposition}
The derived stacks  $\oR\mathring{\cN}_{g,V}'$ and $\bar{W}B\uline{\mathring{\cN}}_{g,V}'$ are equivalent, as are $\oR\mathring{\cN}_{g,V}$ and $\bar{W}B\uline{\mathring{\cN}}_{g,V}$.
\end{proposition}
\begin{proof}
Since $\mathring{\cN}_{g,V}'$ and $\mathring{\cN}_{g,V} $ are formally smooth and homogeneous, the conclusion of Propositions \ref{ulines} and \ref{ulinedg} still applies, meaning that $\bar{W}B\uline{\mathring{\cN}}_{g,V}'$ and $\bar{W}\uline{\mathring{\cN}}_{g,V}$ are the derived $\infty$-stackifications of $B \mathring{\cN}_{g,V}'$ and $ B \mathring{\cN}_{g,V}$ respectively.

It therefore suffices to show that the map
\[
 B \mathring{\cN}_{g,V}' \to \oR\mathring{\cN}_{g,V}'
\]
is another choice of derived $\infty$-stackification. Since the map is an isomorphism on $\pi^0$, it suffices by \cite[Proposition \ref{drep-obs} and Lemma \ref{drep-totcohoc}]{drep} to show that this induces an isomorphism on derived tangent complexes.

At a point $[C] \in \mathring{\cN}_{g,V}'(U)$ corresponding to a log curve $C \to U$, formal smoothness, adjunction and properties of $\cM_g'$ show that the tangent complex of $\mathring{\cN}_{g,V}'$ over $\fLog_V$ is given by
\[
 T_{[C]}( \mathring{\cN}_{g,V}'/ \fLog_V,M)= \oR\HHom_{\O_{\rC}}(\Omega(\rC/\fLog_U), \O_{\rC}\ten_AM)[1],
\]
where $U=\Spec A$ and $M$ is an $A$-module. 

Calculation of the derived tangent complex of $\oR\mathring{\cN}_{g,V}'$ over $\oR\fLog_V$  gives the same answer, essentially because $\cM_g'$ is formally smooth.
\end{proof}

Note that the logarithmic tangent complex $ T_{[C]}( \oR\cN_{g,V}'/V,M)$ can be rewritten as
\[
 T_{[C]}( \oR\mathring{\cN}_{g,V}'/ \oR\fLog_V,M)\simeq \oR\HHom_{\O_{\rC}}(\Omega^{\log}_{C/U}, \O_{\rC}\ten_AM)[1]=\oR\Gamma(\rC, \sT^{\log}_{C/U}\ten_AM)[1].
\]
Beware that this would not be true  for a derived moduli functor based on the approach of \cite{SagaveSchuergVezzosi}, since in that setting the tangent complex would be 
\[
\oR\HHom_{\O_{\rC}}(\bL^G_{C/U}, \O_{\rC}\ten_AM)[1], 
\]
for Gabber's cotangent complex $\bL^G$ as in \cite[\S 8]{olssonlogcot}, and in general $\bL^G_{C/U} \to  \Omega^{\log}_{C/U}$ is not a quasi-isomorphism for log-smooth morphisms $C \to U$.

\subsection{Moduli of log maps}

\subsubsection{The underived  moduli stack}

As in \cite[Definition 3.5.1]{AbramovichMarcusWise}, we now fix a log-flat log stack $X$ over our base log stack $V$, and  consider commutative diagrams
\[
 \begin{CD} 
  C @>>> X \\
@VVV @VVV\\
U@>>> V.     
 \end{CD}
\]
of log stacks as a functor in $U$, where $C \to U$ ranges over proper log-smooth log curves of genus $g$. We will take all the log structures to be locally free. In particular, note that this holds for both the types of morphisms $X \to V$ considered in \cite[\S 2.1]{AbramovichMarcusWise}. For smooth pairs, the log structure on $X$ is free of rank $1$, and that on $V$ is trivial. For acceptable degenerations, the log structure on $X$ is free of rank $2$, with that on $V$ free of rank $1$. 

%

The datum $C \to U$ can be regarded as a map $U \to \cN_{g,V}$ of log stacks, for  $\cN_g$ as above, and then
we seek a map 
\[
 C \to X
\]
of log stacks over $V$. The log structures on $C$ and $X$ over $V$ can be interpreted as morphisms $\rC, \rX \to \fLog_{V,f}$, and then we need a compatible map $\rC \to \fLog_{X,f}= \fLog^1_{V,f}\by_{\fLog_{V,f}}\rX$. Explicitly, this means that the  diagram
\[
 \xymatrix{\rC \ar[r] \ar[dr] & \fLog^1_{V,f}\by_{\fLog_{V,f}}\rX \ar^{\delta_0^* \circ \pr_1}[d]\\
& \fLog_{V,f}
}
\]
must commute.

Now the universal property of $\cD_{g,V}$ gives 
$
\rC= \rU\by_{\fLog_{\cN_{g,V},f}}\mathring{\cD}_{g,V}, 
$
 so we seek a map
\[
   \rU\by_{\fLog_{\cN_{g,V},f}}\mathring{\cD}_{g,V}  \to \fLog_{X,f}
\]
over $\fLog_{V,f}$.

Thus the stack describing this problem is the $\Hom$-stack
\[ 
\mathring{\cN}_{g,V}(X):=\cHom_{\fLog_{\cN_{g,V},f}}(\mathring{\cD}_{g,V}, \fLog_{X,f}\by_{\fLog_{V,f}}\fLog_{\cN_{g,V},f}), 
\] 
where the map $\fLog_{\cN,f} \to \fLog_{V,f}$ is just  given by applying $\fLog$ to the log morphism $\cN_{g,V} \to V$.  

The map $\mathring{\cN}_{g,V}(X) \to \fLog_{\cN_{g,V},f}$ gives us a log structure on the  stack $\mathring{\cN}_{g,V}(X) $  over $\cN_{g,V}$, and we denote the resulting log stack $\cN_{g,V}(X)$

Since stability is a deformation-invariant condition, the log stack of stable maps is an open substack of $\cN_{g,V}(X)$.

\subsubsection{Derived moduli and derived tangent space}

We also wish to understand the derived moduli stack associated to this moduli problem. This should parametrise the $\infty$-groupoid of  commutative diagrams
\[
 \begin{CD} 
  C @>>> X \\
@VVV @VVV\\
U@>>> V.     
 \end{CD}
\]
of derived log stacks, where $C \to U$ ranges over proper log-smooth log curves of genus $g$. Reasoning exactly as above, we see that this problem is  
governed by the derived moduli stack
\[
 \oR\mathring{\cN}_{g,V}(X):= \oR\cHom_{\oR\fLog_{\cN_{g,V}}}(\oR\mathring{\cD}_{g,V}, \oR\fLog_{X}\by^h_{\oR\fLog_{V}}\oR\fLog_{\cN_{g,V}}),
\]
where $\oR\mathring{\cD}_{g,V}:= \bar{W}\uline{ \mathring{\cD}_{g,V}} $ is the derived stackification of the formally smooth stack $\mathring{\cD}_{g,V}$ over $\rV$, and similarly  $\oR\fLog_{\cN_{g,V}}:= \oR\fLog_{\oR\cN_{g,V}} =\bar{W}\uline{ \fLog_{\cN_{g,V},f}}$. 

When $X$ and $V$ are underived log stacks, then  by construction we have
\begin{eqnarray*}
 \pi^0\oR\mathring{\cN}_{g,V}(X)&\simeq& \cHom_{\fLog_{\cN_{g,V},f}}(\mathring{\cD}_{g,V}, \pi^0(\oR\fLog_{X}\by^h_{\oR\fLog_{V}}\oR\fLog_{\cN_{g,V}}))\\
&\simeq& \cHom_{\fLog_{\cN_{g,V},f}}(\mathring{\cD}_{g,V}, \fLog_{X,f}\by_{\fLog_{V,f}}\fLog_{\cN_{g,V},f})\\
&=& \mathring{\cN}_{g,V}(X),
\end{eqnarray*}
the first equivalence coming because $ \oR\mathring{\cD}_{g,V} \to \oR\fLog_{\cN_{g,V}}$ is fibred in underived schemes (curves, in fact). Thus $\oR\mathring{\cN}_{g,V}(X)$ is a derived enhancement of $ \mathring{\cN}_{g,V}(X)$.

Now, properties of derived $\Hom$-stacks give the logarithmic tangent complex  $T_{[(C,m)]}(\oR {\cN}_{g,V}(X)/\oR\cN_{g,V},M )$ of $\oR {\cN}_{g,V}(X)$ over $\oR\cN_{g,V}$
at a map $m \co \rC \to \fLog_X$ as 
\begin{eqnarray*}
  T_{[(C,m)]}(\oR\mathring{\cN}_{g,V}(X)/\oR\fLog_{\cN_{g,V}}, M )&\simeq& \oR\HHom_{\O_{\oR\fLog_{X}}}( \bL^{\oR\fLog_{X}/ \oR\fLog_{V}}, \oR m_* \O_{\rC}\ten M)\\
&\simeq& \oR\HHom_{\O_{\rX}}( \bL^{X/V}, \oR m_* \O_{\rC}\ten M).
\end{eqnarray*}
 Here, we also write $m$ for the induced morphism $\rC \to \rX$, and the second equivalence comes from Remark \ref{triangle} and adjunction.

Substituting for both $ T_{[(C,m)]}(\oR {\cN}_{g,V}(X)/\oR\cN_{g,V},M ) $ and $T_{[C]}( \oR\cN_{g,V}/V,M) $ then gives an exact triangle
\[
\to \oR\HHom_{\O_{\rX}}( \bL^{X/ V}, \oR m_* \O_{\rC}\ten M) \to  T_{[(C,m)]}(\oR {\cN}_{g,V}(X)/V,M ) \to \oR\Gamma(\rC, \sT^{\log}_{C/U}\ten M)[1]
\]

Establishing representability of an open substack of this derived stack (such as that parametrising stable maps) is then just a matter of verifying the finiteness and formal effectiveness conditions of \cite[Theorem \ref{drep-lurierep3}]{drep}, since homotopy-homogeneity follows automatically from our construction, and a derived stack is necessarily a homotopy-preserving hypersheaf.

When $X$ is moreover log smooth over $V$, as happens in all the cases of interest in \cite[\S 2.1]{AbramovichMarcusWise}, $ T_{[(C,m)]}(\oR {\cN}_{g,V}(X)/\oR\cN_{g,V},M ) $ reduces to
\[
 \oR\HHom_{\O_{\rX}}( \Omega^{\log}_{X/V}, \oR m_* \O_{\rC}\ten M)= \oR \Gamma(\rX, \sT^{\log}_{X/V}\ten \oR m_* \O_{\rC}\ten M).
\]
This recovers the relative obstruction theory of  \cite[\S 3.2.3]{AbramovichMarcusWise} as the first cohomology group, and gives the desired virtual relative tangent space on $\cN_{g,V}(X) $.

\bibliographystyle{alphanum}
\bibliography{references}
\end{document}